\documentclass[11pt]{article}

\usepackage[margin=1in]{geometry}
\usepackage{amsmath,amssymb,amsthm}
\usepackage{mathrsfs}
\usepackage{bbm}
\usepackage{hyperref}
\usepackage{graphicx}
\usepackage{color}

\newtheorem{thm}{Theorem}[section]
\newtheorem{lem}[thm]{Lemma}
\newtheorem{prop}[thm]{Proposition}
\newtheorem{defn}[thm]{Definition}
\newtheorem{remark}[thm]{Remark}

\newcommand{\R}{\mathbb{R}}

\newcommand{\PP}{\mathbb{P}}
\newcommand{\dd}{\mathrm{d}}

\begin{document}

\title{A mixed fractional CIR model: positivity and an implicit Euler scheme}
\author{
Chunhao Cai\thanks{School of Mathematics (Zhuhai), Sun Yat-sen University, Guangzhou, China.
Email: caichh9@mail.sysu.edu.cn}
\and
Cong Zhang\thanks{School of Management Science, Zhejiang University, Hangzhou, China.
Email: zcong@zju.edu.cn}
}
\date{\today}

\maketitle

\begin{abstract}
We consider a Cox--Ingersoll--Ross (CIR) type short rate model driven by a mixed fractional Brownian motion. Let
$M=B+B^H$ be a one-dimensional mixed fractional Brownian motion with Hurst index $H>1/2$, and let
$\mathbf{M}=(M,\mathbb{M}^{\mathrm{It\hat{o}}})$ denote its canonical It\^o rough path lift. We study the rough
differential equation
\begin{equation}\label{eqn1}
\dd r_t = k(\theta-r_t)\,\dd t + \sigma\sqrt{r_t}\,\dd\mathbf{M}_t,\qquad r_0>0,
\end{equation}
and show that, under the Feller condition $2k\theta>\sigma^2$, the unique rough path solution exists and is almost surely
strictly positive for all times. Our approach relies on an It\^o type formula for rough paths, together with refined
pathwise estimates for the mixed fractional Brownian motion, including L\'evy's modulus of continuity for the Brownian
part and a law of the iterated logarithm for the fractional component. As a consequence, the positivity property of the
classical CIR model extends to this non-Markovian rough path setting. We also establish the convergence, in the uniform
norm, of an implicit Euler scheme for the associated singular equation obtained by a square-root transformation.
\end{abstract}

\section{Introduction}
Short-rate models of Cox--Ingersoll--Ross (CIR) type play a central role in the modeling of interest rate dynamics in
mathematical finance, as they combine mean-reversion with a strictly positive state space; see, for instance, the
seminal work of Cox, Ingersoll and Ross~\cite{CIR85}. In its classical form, the CIR model is driven by a Brownian
motion and hence inherits the Markov and semimartingale properties of the driving noise. Empirical evidence for
long-range dependence, roughness and memory effects in financial time series has motivated various extensions of the
CIR dynamics in which the Brownian motion is replaced or complemented by fractional or mixed fractional Brownian
motions~\cite{cheridito2001,hu2008}. The purpose of this paper is to analyze a CIR-type short rate model driven by a
mixed fractional Brownian motion within the rough path framework, focusing on pathwise positivity of the short rate and on the convergence of a natural implicit Euler approximation scheme. The rough-path approach used here follows the foundations of Lyons’ theory and its extensions; see Lyons~\cite{lyons1998} and Friz \& Victoir~\cite{friz-victoir}, and for results on differential equations driven by fractional Brownian motion see Nualart \& R\u{a}\c{s}canu~\cite{nualart-rascanu}.

We consider the one-dimensional rough differential equation
\begin{equation}\label{eq:intro-mfcir}
\dd r_t  = k(\theta- r_t)\,\dd t + \sigma \sqrt{r_t}\,\dd\mathbf{M}_t,\qquad r_0>0,
\end{equation}
where $k,\theta,\sigma>0$, and $\mathbf{M}=(M,\mathbb{M}^{\mathrm{It\hat{o}}})$ is the It\^o rough path lift of the
mixed fractional Brownian motion $M=B+B^H$. Here $B$ is a standard Brownian motion, $B^H$ is an independent fractional
Brownian motion with Hurst parameter $H>1/2$, and
\[
\mathbb{M}^{\mathrm{It\hat{o}}}_{s,t}=\int_s^t M_{s,r}\otimes \dd M_r
\]
is the It\^o iterated integral. The process $\mathbf{M}$ is a canonical example of a Gaussian rough path.

For the classical CIR model driven by a Brownian motion $W$,
\[
\dd r_t = k(\theta-r_t)\,\dd t + \sigma\sqrt{r_t}\,\dd W_t,\qquad r_0>0,
\]
it is well known that the solution remains nonnegative, and even strictly positive, under the Feller condition
$2k\theta>\sigma^2$ (see e.g.\ Feller's classification of boundaries and the monograph of Karatzas and Shreve).
Our goal is to show that a similar positivity property holds for the mixed fractional model \eqref{eqn1}, in a rough path
framework.

The main steps of the analysis are as follows. First, we recall the notion of bracket for a rough path and prove an
It\^o type formula for the mixed fractional rough path $\mathbf{M}$. This allows us to apply the square-root
transformation $z_t=\frac{2}{\sigma}\sqrt{r_t}$ up to the (hypothetical) hitting time of zero, and we obtain a singular
mean-reverting equation for $z$. Second, we use pathwise estimates for $M=B+B^H$, namely L\'evy's modulus for $B$ and a
law of the iterated logarithm for $B^H$, to show that the singular drift in the $z$-equation prevents $z$ from hitting
zero in finite time. This implies that $r_t>0$ almost surely for all $t\ge0$. Finally, we study an implicit Euler scheme
for the singular equation and prove its convergence in the uniform norm, following ideas of Marie~\cite{marie2015}.

\section{An It\^o formula for the mixed fractional rough path}

The first ingredient in our analysis is an It\^o type formula tailored to the mixed fractional rough path
$\mathbf{M}=(M,\mathbb{M}^{\mathrm{It\hat{o}}})$. To this end, we recall the notion of the bracket of a rough path and
establish a general identity that expresses the bracket in terms of the first and second levels of the rough path. We
then specialize this identity to the case of the mixed fractional Brownian motion $M=B+B^H$ and compute the bracket of
its canonical It\^o lift. This provides the key tool needed to perform a square-root transformation of the mixed
fractional CIR equation in a purely pathwise fashion.

In this section, we consider rough paths $\mathbf{X}=(X,\mathbb{X}) \in \mathscr{C}^{\alpha}$ with
$\alpha \in (\frac{1}{3}, \frac{1}{2}]$ unless otherwise noted. Let $V$ be a Banach space. For $x \in V \otimes V$, we
introduce the symmetric part of $x$:
\[
\operatorname{Sym}(x) = \frac{1}{2}\left(x + x^T\right). 
\]
Then the bracket of the rough path will be defined as: 
\begin{defn}
Let $\mathbf{X}=(X,\mathbb{X}) \in \mathscr{C}^{\alpha}$ be a rough path. The bracket of $\mathbf{X}$ is the path
$[\mathbf{X}]: [0,T] \to \R^{d \times d}$ given by
\begin{equation}
[\mathbf{X}]_t:=(X_{0,t}\otimes X_{0,t}) - 2\operatorname{Sym}(\mathbb{X}_{0,t}).
\end{equation}
\end{defn}

\begin{remark}
\label{rk1}
By Chen's relation for $\mathbb{X}$, one checks that, for all $(s,t) \in \Delta_{[0, T]}$,
\begin{equation}
[\mathbf{X}]_{s,t}=(X_{s,t}\otimes X_{s,t}) - 2\operatorname{Sym}(\mathbb{X}_{s,t}). 
\end{equation}    
\end{remark}

We now specialize to the mixed fractional Brownian motion.

\begin{lem}\label{lem:bracket-m}
Let $M = B + B^H$ be a one-dimensional mixed fractional Brownian motion, given as the sum of a standard Brownian motion
$B$ and an independent fractional Brownian motion $B^H$ with Hurst index $H > \frac{1}{2}$. Consider
$\mathbf{M}=(M,\mathbb{M}^{\mathrm{It\hat{o}}})$ where
\begin{equation}
\mathbb{M}^{\mathrm{It\hat{o}}}_{s,t}:=\int_{s}^{t}M_{s,r}\otimes\,\dd M_r.    
\end{equation}
Then for any finite $t > 0$,
\begin{equation}
[\mathbf{M}]_{t} = t.
\end{equation}    
\end{lem}

\begin{proof}
In the one-dimensional case we have $\operatorname{Sym}(\mathbb{M}_{s,t}) = \mathbb{M}_{s,t}$.

For brevity, write $\mathbb{M}$ for $\mathbb{M}^{\mathrm{It\hat{o}}}$. Fix $t>0$, and let
$\pi: 0 = t_0 < t_1 < \dots < t_n = t$ be a partition of $[0,t]$ with mesh size
$|\pi| := \max_{i}\left| t_{i+1} - t_{i}\right|$. By Remark~\ref{rk1},
\[
[\mathbf{M}]_t = \lim_{|\pi| \to 0}\sum_{i = 0}^{n-1} [\mathbf{M}]_{t_i, t_{i+1}},
\]
where
\[
[\mathbf{M}]_{t_i, t_{i+1}} = (M_{t_i, t_{i+1}})^2 - 2\mathbb{M}_{t_i, t_{i+1}}.
\]
Since $M = B + B^H$, we obtain
\begin{equation}\label{eq:key3}
\begin{aligned}
[\mathbf{M}]_{t_i, t_{i+1}}
&= B_{t_i, t_{i+1}}^2 + 2 B_{t_i, t_{i+1}} B^H_{t_i, t_{i+1}} + (B^H_{t_i, t_{i+1}})^2 \\
&\quad - 2\mathbb{B}_{t_i, t_{i+1}} - 2\mathbb{B}^H_{t_i, t_{i+1}} - 2\int_{t_i}^{t_{i+1}}B^H_{t_i, r}\,\dd B_r
       - 2\int_{t_i}^{t_{i+1}}B_{t_i, r}\,\dd B^H_r,
\end{aligned}
\end{equation}
where
\[
\mathbb{B}_{s,t}:=\int_s^t B_{s,r}\,\dd B_r,\qquad \mathbb{B}^H_{s,t}:=\int_s^t B^H_{s,r}\,\dd B^H_r.
\]
We group the non-Brownian terms as
\[
\mathrm{I}_i := 2 B_{t_i, t_{i+1}} B^H_{t_i, t_{i+1}} + (B^H_{t_i, t_{i+1}})^2,
\]
and
\[
\mathrm{II}_i := \mathbb{B}^H_{t_i, t_{i+1}}+\int_{t_i}^{t_{i+1}}B^H_{t_i, r}\,\dd B_r
+\int_{t_i}^{t_{i+1}}B_{t_i, r}\,\dd B^H_r.
\]
Then \eqref{eq:key3} can be rewritten as
\[
[\mathbf{M}]_{t_i, t_{i+1}} = B_{t_i, t_{i+1}}^2 - 2\mathbb{B}_{t_i, t_{i+1}}
+ \mathrm{I}_i - 2\mathrm{II}_i.
\]

Since $B$ and $B^H$ are almost surely H\"older continuous of any order strictly less than $\frac{1}{2}$ and $H$,
respectively, there exists $\varepsilon \in (0,\tfrac{H}{2}-\tfrac{1}{4})$ such that $B$ is
$(\tfrac12-\varepsilon)$-H\"older and $B^H$ is $(H-\varepsilon)$-H\"older. Hence
\[
|B_{t_i, t_{i+1}}|\lesssim |t_{i+1}-t_i|^{\frac12-\varepsilon},\qquad
|B^H_{t_i, t_{i+1}}|\lesssim |t_{i+1}-t_i|^{H-\varepsilon},
\]
and therefore
\begin{equation}\label{eqn1rp}
\mathrm{I}_i = O\big((t_{i+1}-t_i)^{H+\frac12-2\varepsilon}\big).
\end{equation}

For the integrals in $\mathrm{II}_i$ we use Young integration: since
$(\tfrac12-\varepsilon)+(H-\varepsilon)>1$, the Young integrals $\int B^H\,\dd B$ and $\int B\,\dd B^H$ exist and satisfy
\[
\left|\int_{t_i}^{t_{i+1}}B^H_{t_i,r}\,\dd B_r\right|
\lesssim \|B^H\|_{H-\varepsilon,[t_i,t_{i+1}]}\,\|B\|_{\frac12-\varepsilon,[t_i,t_{i+1}]}\,
|t_{i+1}-t_i|^{H+\frac12-2\varepsilon},
\]
and similarly for $\int B_{t_i,r}\,\dd B^H_r$. Moreover,
\[
\mathbb{B}^H_{t_i, t_{i+1}} = O\big((t_{i+1}-t_i)^{2H-2\varepsilon}\big),
\]
with $2H-2\varepsilon>H+\tfrac12-2\varepsilon$. Altogether we obtain
\begin{equation}\label{eqn2rp}
\mathrm{II}_i = O\big((t_{i+1}-t_i)^{H+\frac12-2\varepsilon}\big).
\end{equation}

Combining \eqref{eqn1rp} and \eqref{eqn2rp}, we arrive at
\begin{equation}\label{eq:key1}
\sum_{i=0}^{n-1}\left( \mathrm{I}_i - 2\mathrm{II}_i\right)
= O\Big(\sum_{i=0}^{n-1}(t_{i+1}-t_i)^{H+\frac12-2\varepsilon}\Big)
\longrightarrow 0\quad\text{as }|\pi|\to 0,
\end{equation}
since $H+\tfrac12-2\varepsilon>1$.

On the other hand, by the usual It\^o calculus for $B$ we know that
\[
\lim_{|\pi|\to 0}\sum_{i=0}^{n-1}\left(B_{t_i,t_{i+1}}^2-2\mathbb{B}_{t_i,t_{i+1}}\right)=t
\]
almost surely. Together with \eqref{eq:key1} this yields
\[
[\mathbf{M}]_t = \lim_{|\pi|\to 0}\sum_{i=0}^{n-1}[\mathbf{M}]_{t_i,t_{i+1}} = t,
\]
which completes the proof.
\end{proof}

Now we will introduce the It\^o formula for this rough pats
\begin{prop}\label{pro1}
Let $\mathbf{M}=(M,\mathbb{M})\in\mathscr{C}^{\alpha}$ be a rough path, and let $f\in C^{3}(\R)$. Then
\begin{equation}\label{eq:ito-rough}
f(M_{T})=f(M_{0})+\int_{0}^{T}Df(M_{u})\,\dd\mathbf{M}_{u}
+\frac{1}{2}\int_{0}^{T}D^{2}f(M_{u})\,\dd[\mathbf{M}]_{u}.
\end{equation}
\end{prop}

\begin{proof}
Assume without loss of generality that $f\in C_{b}^{3}$. Since $Df\in C_{b}^{2}$, the pair
$(Df(M),D^{2}f(M))$ is a controlled path with respect to $M$. For $0\le s<t\le T$ we have the Taylor expansion
\begin{equation}
f(M_{t})-f(M_{s})=Df(M_{s})M_{s,t}+D^{2}f(M_{s})\mathbb{M}_{s,t}+R_{s,t},
\end{equation}
where
\begin{equation}
R_{s,t}:=\int_{0}^{1}\int_{0}^{1}D^{3}f\big(M_{s}+r_{1}r_{2}M_{s,t}\big)
(M_{s,t}\otimes M_{s,t})\,r_{1}\,\dd r_{2}\,\dd r_{1}.
\end{equation}
Thus
\[
|R_{s,t}|\leq\|f\|_{C_{b}^{3}}|M_{s,t}|^{3}\leq\|f\|_{C_{b}^{3}}\|M\|_{\alpha}^{3}|t-s|^{3\alpha},
\]
so that, for any sequence of partitions $\pi$ with mesh $|\pi|\to 0$,
\[
\sum_{[s,t]\in\pi}|R_{s,t}|\longrightarrow 0.
\]

Since the Hessian $D^{2}f(M_{s})$ is symmetric, it kills the antisymmetric part of $\mathbb{M}_{s,t}$, that is,
\[
D^{2}f(M_{s})\mathbb{M}_{s,t}=D^{2}f(M_{s})\operatorname{Sym}(\mathbb{M}_{s,t}).
\]
Using Definition~\ref{rk1} we have
\[
\operatorname{Sym}(\mathbb{M}_{s,t})=\tfrac12\big(M_{s,t}\otimes M_{s,t}-[\mathbf{M}]_{s,t}\big),
\]
and therefore
\[
D^{2}f(M_{s})\mathbb{M}_{s,t}
=\frac12D^{2}f(M_{s})(M_{s,t}\otimes M_{s,t})-\frac12D^{2}f(M_{s})[\mathbf{M}]_{s,t}.
\]
Summing over a partition and passing to the limit $|\pi|\to 0$ gives \eqref{eq:ito-rough}.
\end{proof}

Now return to the mixed fractional CIR equation \eqref{eqn1}. Fix $\varepsilon_0>0$ and define the stopping time
\[
\tau^{r}_{\varepsilon_0} := \inf\{s > 0: r_s = \varepsilon_0\}.
\]
For any $0\le s \le t \le \tau^r_{\varepsilon_0}$, define
\[
z_t := \frac{2}{\sigma}\sqrt{r_t}.
\]
On $[\varepsilon_0,\infty)$ the function $f(x)=\frac{2}{\sigma}\sqrt{x}$ is $C^3_b$, so by
Proposition~\ref{pro1} and Lemma~\ref{lem:bracket-m} we obtain
\begin{equation}
\begin{aligned}
z_{t} =& z_s + \int_{s}^{t} \frac{1}{\sigma \sqrt{r_u}}k(\theta- r_u)\,\dd u 
+ \int_{s}^{t} \frac{1}{\sigma \sqrt{r_u}} \sigma \sqrt{r_u}\,\dd\mathbf{M}_{u} \\
&\quad + \frac{1}{2} \int_{s}^{t}-\frac{1}{2}\frac{1}{\sigma}r_u^{-\frac{3}{2}} \sigma^2 r_u\,\dd[\mathbf{M}]_{u}\\
=&z_s + \int_{s}^{t} \left(\frac{2k\theta}{\sigma^2 z_u}- \frac{k z_u}{2}\right)\,\dd u 
+ M_{s,t} -\int_{s}^{t}\frac{1}{2z_u }\,\dd u\\ 
=&z_s + \int_{s}^{t} \left(\frac{4k\theta-\sigma^2}{2\sigma^2 z_u}- \frac{k z_u}{2}\right) \,\dd u + 
M_{s,t}.
\end{aligned}    
\end{equation}
Define
\begin{equation}
m :=
\frac{2 k\theta-\sigma^2}{\sigma^2},
\end{equation}
we obtain the singular equation
\begin{equation}
\label{eqnsingular}
\dd z_t = \left[(m+\frac{1}{2})z_t^{-1} - \frac{k}{2} z_t\right] \dd t +\dd M_t. 
\end{equation}
A standard localization argument (using the stopping times $\tau^r_{\varepsilon_0}$ or a smooth approximation of
$\sqrt{\cdot}$ near zero) shows that \eqref{eqnsingular} actually holds for all $t$ strictly smaller than the first
hitting time of $0$ by $z$.

\section{Positivity of the mixed fractional CIR process}\label{sec:positivity}

In this section we turn to the positivity of the mixed fractional CIR process. The square-root transformation developed
in the previous section shows that the problem reduces to the analysis of a singular mean-reverting equation for the
process $z_t=\tfrac{2}{\sigma}\sqrt{r_t}$. The main difficulty stems from the fact that the drift in the $z$-equation
blows up near zero, while the driving signal $M=B+B^H$ is rough and non-Markovian. Our strategy is to combine sharp
modulus-of-continuity estimates for the Brownian motion $B$ with a law of the iterated logarithm for the fractional
component $B^H$, and to exploit these pathwise estimates in order to rule out the possibility that $z$ hits zero in
finite time.

Before proving the positivity of the solution $(r_t)_{t\ge0}$ to \eqref{eqn1}, we recall two classical pathwise results
for Brownian motion and fractional Brownian motion. The first one is the L\'evy's modulus of continuity, see e.g.\ \cite{1989Brownian}

\begin{lem}
\label{lem1}
Let $B$ be a one-dimensional standard Brownian motion. Put $g(t) := \sqrt{2t\log(1/t)}$. Then
\begin{equation}
\label{eqn2}
\PP\left[ \limsup_{\delta \to 0^+} \frac{1}{g(\delta)}\sup_{\tiny \begin{aligned}
0 \le& s < t < 1 \\    
& t - s \le \delta
\end{aligned}
}\left\lvert B_t - B_s \right\rvert = 1 
\right] = 1.   
\end{equation}
\end{lem}

Then the second one is the Law of the iterated logarithm for fBm, see \cite{hu2008,1995On}: 
\begin{lem}
\label{lem2}
Let $B^H$ be a one-dimensional fractional Brownian motion with Hurst index $H\in(0,1)$. Then
\begin{equation}
\label{eqn3}
\limsup_{t \to 0^+}\frac{\left\lvert B^H_t\right\rvert }{t^H\sqrt{\log\log t^{-1}}} = c_H
\end{equation}
with probability one, where $c_H \in (0,\infty)$ is a suitable constant depending only on $H$.
\end{lem}

With these two lemmas We can state the main positivity result:
\begin{thm}
\label{thm1}
Consider \eqref{eqn1} with $r_0 > 0$. Let $k,\theta>0$, $H > \frac{1}{2}$, and assume that
\begin{equation}
\label{eqn4}
2 k\theta > \sigma^2,  
\end{equation}
which is the classical Feller condition. Then the process $(r_t)_{t\ge0}$ defined by \eqref{eqn1} is strictly positive
almost surely.
\end{thm}
\begin{proof}
We work with the transformed process $z_t = \frac{2}{\sigma}\sqrt{r_t}$, which satisfies \eqref{eqnsingular} up to its
first hitting time of $0$. Note that $m > 0$ by \eqref{eqn4}. Define
\[
\tau := \inf\{s > 0: z_s = 0\}.
\]
We will show that $\PP(\tau<\infty)=0$. Suppose $\tau$ is finite, by Lemma \ref{lem1},for any $\varepsilon > 0$ such that $2\alpha_0 m +1> \left(1 + \varepsilon\right)^2$ for some $\alpha_0 \in (0, 1)$, there always exists $\delta(\omega, \varepsilon)$ such that 
\begin{equation}
\label{eqnbmloc}
\left\lvert B_{\tau} - B_t\right\rvert \le (1+\varepsilon)\sqrt{2\delta_1(t)\log \frac{1}{\delta_1(t)}} \quad a.s. 
\end{equation}
for all $\delta_1(t) := \tau -t \in (0, \delta)$.

Assume by contradiction that $\tau<\infty$ with positive probability. Fix $\omega$ in a full-measure event on which all
the pathwise estimates below hold, and work pathwise. For $t<\tau$, integrating \eqref{eqnsingular} from $t$ to $\tau$
yields
\begin{equation}\label{eq:z-int}
z_t + \int_{t}^{\tau}\left((m+\frac{1}{2}) z_s^{-1} - \frac{k}{2}z_s \right) \,\dd s = M_t - M_{\tau},   
\end{equation}
where $M=B+B^H$.

Fix $\gamma^\ast\in(0,H-\tfrac12)$. Consider $\tau - \delta_2 < t < \tau$ where $\tau_2$ is small enough so that we can always ensure $(m+\frac{1}{2}) z_s^{-1} - \frac{k}{2}z_s > 0$. For 
\begin{equation}
z_t + \int_{t}^{\tau}\left((m+\frac{1}{2}) z_s^{-1} - \frac{k}{2}z_s \right) \,ds = M_t - M_{\tau},   
\end{equation}
since $M$ has H\"older continuous paths of any order $\alpha< 1/2$, there
exist random constants $C(\omega, \gamma^\ast)>0$ and $\delta_2(\omega)>0$ such that
\begin{equation}\label{eq:z-upper}
z_t \le C(\omega, \gamma^\ast)|\tau - t|^{\frac{1}{2} - \gamma^\ast},\qquad
\text{for all }t\in(\tau-\delta_2,\tau).
\end{equation}
In particular, $z_t\to0$ as $t\uparrow\tau$.

Fix $\xi>0$ small and define $\tau_\xi := \sup\{s \in (\tau- \min\{\delta, \delta_2\}, \tau): z_s = \xi \}$, so that we have 
\begin{equation}
-\xi = \int_{\tau_\xi}^{\tau} \left((m+\frac{1}{2}) z_s^{-1} - \frac{k}{2}z_s \right) \,ds + M_{\tau} - M_{\tau_\xi}.    
\end{equation}

Since we Suppose
\begin{equation}
M_t = B_t + B_t^H,   
\end{equation}
we have
\begin{equation}
-\xi = \int_{\tau_\xi}^{\tau} \left((m+\frac{1}{2}) z_s^{-1} - \frac{k}{2}z_s \right) \,ds + B_{\tau} - B_{\tau_\xi} + B^H_{\tau} - B^H_{\tau_\xi}.     
\end{equation}

Then there $\exists \alpha \in (\alpha_0, 1)$ that 
\begin{equation}
\begin{aligned}
\left\lvert B_{\tau} - B_{\tau_\xi}\right\rvert + \left\lvert B^H_{\tau} - B^H_{\tau_\xi} \right\rvert \ge& \int_{\tau_\xi}^{\tau} (1-\alpha)m z_s^{-1}  \,ds \\&+ \left((\alpha m+\frac{1}{2}) \xi^{-1} - \frac{k}{2}\xi \right)(\tau - \tau_\xi) + \xi,      
\end{aligned}
\end{equation}
notice $\tau - \tau_\xi < \delta$, almost surely, we have the following
\begin{equation}
\begin{aligned}
&(1+\varepsilon)\sqrt{2\delta_1(\tau_\xi)\log \frac{1}{\delta_1(\tau_\xi)}} + \left\lvert B^H_{\tau} - B^H_{\tau_\xi} \right\rvert \ge \int_{\tau_\xi}^{\tau} (1-\alpha)m z_s^{-1}  \,ds + \left((\alpha m+\frac{1}{2}) \xi^{-1} - \frac{k}{2}\xi \right)(\tau - \tau_\xi) + \xi,\\
&\left\lvert B^H_{\tau} - B^H_{\tau_\xi}\right\rvert \ge \int_{\tau_\xi}^{\tau} (1-\alpha)m z_s^{-1}  \,ds + \left((\alpha m+\frac{1}{2}) \xi^{-1} - \frac{k}{2}\xi \right)(\tau - \tau_\xi) -(1+\varepsilon)\sqrt{2\delta_1(\tau_\xi)\log \frac{1}{\delta_1(\tau_\xi)}} + \xi.
\end{aligned}
\end{equation}

From now to discuss $\left((\alpha m+\frac{1}{2}) \xi^{-1} - \frac{k}{2}\xi \right)(\tau - \tau_\xi) -(1+\varepsilon)\sqrt{2\delta_1(\tau_\xi)\log \frac{1}{\delta_1(\tau_\xi)}} + \xi$, notice
$\exists \tilde{\varepsilon} > 0 $ such that $\forall \varepsilon < \tilde{\varepsilon}$, $\frac{(\alpha m+\frac{1}{2})}{\varepsilon} - \frac{k}{2} \varepsilon > \frac{\alpha_0 m+\frac{1}{2}}{\varepsilon}$
(Note that $\alpha_0$ has been fixed before \eqref{eqnbmloc} and $\alpha \in (\alpha_0, 1)$). Consider
\begin{equation}
\begin{aligned}
f(x) &= \frac{2\alpha_0 m+1}{2\xi} x - (1+ \varepsilon) \sqrt{2x\log \frac{1}{x}} + \xi,\\
f_\gamma(x) &= \frac{2\alpha_0 m+1}{2\xi} x - (1+ \varepsilon) \sqrt{2}x^\gamma + \xi,    
\end{aligned}
\end{equation}
 where $\gamma \in (0, \frac{1}{2})$. We know

\begin{align*}
f_\gamma^\prime(x) = \frac{2\alpha_0 m+1}{2\xi} - (1+\varepsilon)\sqrt{2}\gamma x^{\gamma-1},   
\end{align*}
let $f_\gamma^\prime(x_0) = 0$, then $x_0 = \left(\frac{2\alpha_0 m+1}{2\sqrt{2}(1+\varepsilon)\xi \gamma}\right)^\frac{1}{\gamma-1} $,

\begin{align*}
f_\gamma(x_0) =& (1+\varepsilon)\sqrt{2}(\gamma - 1)x_0^\gamma + \xi\\
=&\left(\sqrt{2}(1 + \varepsilon)\right)^{-\frac{1}{\gamma-1}}\left(\frac{2\gamma}{ 2\alpha_0 m + 1}\right)^{-\frac{\gamma}{\gamma-1}}(\gamma-1) \xi^{-1 - \frac{1}{\gamma-1}} + \xi\\
=&g(\gamma).
\end{align*}
Note when $\gamma \to \frac{1}{2}^-$, 
\begin{align*}
g(\gamma) \to \left(1 - \frac{(1+\varepsilon)^2}{2\alpha_0 m+1}\right)\xi > 0,       
\end{align*}
since $g(\gamma)$ is continuous in $(0, \frac{1}{2})$, so $\exists \gamma_0 > 0$, $\forall \gamma \in (\gamma_0, \frac{1}{2})$, $g(\gamma) > 0$. 
Fix any $\gamma_1 \in (\gamma_0, \frac{1}{2})$. Then we have 
\begin{equation}
\begin{aligned}
\left\lvert B^H_{\tau} - B^H_{\tau_\xi}\right\rvert \ge& \int_{\tau_\xi}^{\tau} (1-\alpha)m z_s^{-1}  \,ds + f(\tau-\tau_\xi) - f_{\gamma_1}(\tau-\tau_\xi) + f_{\gamma_1}(\tau-\tau_\xi) \\
\ge& \int_{\tau_\xi}^{\tau} (1-\alpha)m z_s^{-1}  \,ds + f(\tau-\tau_\xi) - f_{\gamma_1}(\tau-\tau_\xi), \quad a.s.     
\end{aligned}
\end{equation}

by Lemma~\ref{lem2}, for $\xi$ small enough, there exists a constant $C_H^\ast>0$
such that we have
\begin{equation}\label{eq:LIL-xi}
|B^H_\tau-B^H_{\tau_\xi}|
\le C_H^\ast\,(\tau-\tau_\xi)^H\sqrt{\log\log\tfrac{1}{\tau-\tau_\xi}}.
\end{equation}
It follows that
\begin{equation}
\begin{aligned}
C_H^\ast \left(\tau - \tau_\xi \right) ^{H}\sqrt{\log\log (\tau - \tau_\xi)^{-1}} &\ge C (1-\alpha)m \int_{\tau_\xi}^{\tau}(\tau - s)^{-\left(\frac{1}{2} - \gamma^\ast\right) }\,ds+f(\tau-\tau_\xi) - f_{\gamma_1}(\tau-\tau_\xi)\\
&\ge  C (1-\alpha)m \left(\frac{1}{2} + \gamma^\ast\right)^{-1}(\tau - \tau_\xi)^{\frac{1}{2}+ \gamma^\ast} + f(\tau-\tau_\xi) - f_{\gamma_1}(\tau-\tau_\xi),\quad a.s. 
\end{aligned}
\end{equation}
and notice for any $\gamma >0$, $f(x) - f_\gamma(x) = (1+\varepsilon)\sqrt{2}\left(x^\gamma - \sqrt{x\log \frac{1}{x}}\right) \ge 0$ for any $x$ sufficiently small, 
so let $\xi \to 0$, then $\tau_\xi \to \tau$, we only need to consider
\begin{equation}
\label{eqnfinal}
\left(\tau - \tau_\xi \right) ^{H-\frac{1}{2}-\gamma^\ast}\sqrt{\log\log (\tau - \tau_\xi)^{-1}} \ge  C_f(1-\alpha)m \left(\frac{1}{2} + \gamma^\ast\right)^{-1}, \quad a.s.   
\end{equation}
as $\xi \to 0$, then it's clear that as $\tau_\xi \to \tau$,
\begin{equation}
\left(\tau - \tau_\xi \right) ^{H-\frac{1}{2}-\gamma^\ast}\sqrt{\log\log (\tau - \tau_\xi)^{-1}} \to 0, \quad \text{a.s.}  
\end{equation}
since $H-\frac{1}{2} > \gamma^\ast$, then we have
\begin{equation}
0 \ge C_f(1-\alpha)m \left(\frac{1}{2} + \gamma^\ast\right)^{-1} > 0.    
\end{equation}
It follows that there is a contradiction from the limits of eq.\eqref{eqnfinal}.
\end{proof}

\section{Convergence of the implicit Euler scheme for the singular equation}
Having established the pathwise positivity of the mixed fractional CIR process, we finally address the numerical
approximation of the transformed process $z$. A natural candidate is the implicit Euler scheme associated with the
singular equation satisfied by $z$. The singularity of the drift at the origin makes the analysis of this scheme
nontrivial, and a stability argument is required to control the discrete trajectories uniformly in time. Building on the
framework developed by Marie~\cite{marie2015}, we derive pathwise estimates for the discrete scheme and prove its
convergence to the exact solution in the uniform norm on compact time intervals.
\subsection{Implicit Euler scheme and its convergence for the singular equation}
Recall the singular equation
\begin{equation}\label{eqnsingular-again}
\dd z_t = \Big[(m+\tfrac12)z_t^{-1} - \tfrac{k}{2} z_t\Big] \,\dd t + \dd M_t, 
\end{equation}
where $M=B+B^H$ is the mixed fractional Brownian motion and 
\[
m = \frac{2k\theta - \sigma^2}{\sigma^2}>0.
\]
Fix a time horizon $T>0$, and let $(t_k^n)_{k=0,\dots,n}$ be the uniform partition of $[0,T]$ given by
\[
t_k^n := \frac{kT}{n},\qquad \Delta t := t_{k+1}^n - t_k^n = \frac{T}{n}.
\]
The implicit Euler scheme associated to \eqref{eqnsingular-again} and 
to $(t_0^n, t_1^n, \dots, t_n^n)$ is 
\begin{equation}
\label{eqnimplicit}
z^n_{k+1} = z^n_{k} +  b(z^n_{k+1})\,\Delta t + M_{t^n_{k+1}}- M_{t^n_k},\qquad k=0,\dots,n-1,
\end{equation}
with
\[
b(x) = \Big(m+\frac{1}{2}\Big)x^{-1} - \frac{k}{2} x,\qquad z^n_0 := z_0 > 0.
\]

The proof of convergence is inspired by Marie~\cite{marie2015}.

\begin{remark}
Although \eqref{eqnsingular-again} does not satisfy Assumption~1.1.(4) of~\cite{marie2015}, 
the fact that it satisfies Assumptions~1.1.(1)--1.1.(3) makes the convergence 
argument in Marie~\cite{marie2015} applicable in our setting.    
\end{remark}

We define the continuous-time interpolation $(z_t^n)_{t\in[0,T]}$ by
\begin{equation}
z_t^n = \sum_{k=0}^{n-1} \left[ z_k^n + \frac{z_{k+1}^n - z_k^n}{t_{k+1}^n - t_k^n}(t - t_k^n) \right]
\mathbbm{1}_{(t_k^n, t_{k+1}^n]}(t),\qquad t\in[0,T].
\end{equation}
With the convention $z_0^n=z_0$, this defines a continuous process on $[0,T]$.

We also use the notation
\[
\|f\|_{\infty,T} := \sup_{t\in[0,T]}|f_t|,\qquad
\|f\|_{\alpha,T} := \sup_{0\le s<t\le T}\frac{|f_t-f_s|}{|t-s|^\alpha}
\]
for the uniform and H\"older norms on $[0,T]$.

\begin{thm}
\label{thm:euler-convergence}
Assume that $2 k\theta > \sigma^2$ and $k > 0$. Then the implicit Euler scheme \eqref{eqnimplicit} is well-defined and has a unique positive solution. Moreover, let $\alpha:=\frac12-\varepsilon$ with any $\varepsilon\in(0,1/2)$. There exists a constant $C > 0$ depending on $T$, $m$, $k$, and the sample path of $M$ such that
\begin{equation}
\|z^n - z\|_{\infty,T} \leq C\, n^{-\alpha}
= C\,n^{-\left(\frac{1}{2}-\varepsilon\right)}.
\end{equation}
\end{thm}

\begin{proof}
The proof follows the approach of Marie~\cite{marie2015}.

\medskip\noindent First, we show that the scheme is well-defined and remains strictly positive.

For each $k$, equation \eqref{eqnimplicit} is an implicit equation for $z_{k+1}^n$. Define the function
\begin{equation}
\varphi(x) = x - b(x)\,\Delta t - z_k^n - \Delta M_k,
\end{equation}
where $\Delta t = T/n$ and $\Delta M_k = M_{t_{k+1}^n} - M_{t_k^n}$. 

We compute the derivative:
\begin{equation}
\varphi'(x) = 1 - b'(x)\,\Delta t
= 1 + \left[\Big(m+\frac{1}{2}\Big)x^{-2} + \frac{k}{2}\right] \Delta t > 0 \quad \text{for all } x > 0,
\end{equation}
so $\varphi$ is strictly increasing on $(0,\infty)$. Moreover, $\varphi(x) \to -\infty$ as $x \to 0^+$ and
$\varphi(x) \to +\infty$ as $x \to +\infty$. Therefore, by the intermediate value theorem, there exists a unique positive
solution $z_{k+1}^n > 0$. Inductively, the scheme is well-defined and remains strictly positive.

\medskip\noindent Next, we show that the scheme is uniformly bounded on $[0,T]$.

Fix $k \in \{1,\ldots,n\}$ and define
\begin{equation}
n(z_0,k) := \max\{i \in \{0,\ldots,k\}: z_i^n \leq z_0\}.
\end{equation}
If $n(z_0,k) = k$, then $0 < z_k^n \leq z_0$. If $n(z_0,k) < k$, then
\begin{equation}
\begin{aligned}
z_k^n - z_{n(z_0,k)}^n 
&= \sum_{i=n(z_0,k)}^{k-1} (z_{i+1}^n - z_i^n) \\
&= \sum_{i=n(z_0,k)}^{k-1} \Big( b(z_{i+1}^n)\Delta t + \Delta M_i\Big) \\
&= \Delta M_{n(z_0,k),k} + \sum_{i=n(z_0,k)}^{k-1} b(z_{i+1}^n)\Delta t,
\end{aligned}
\end{equation}
where $\Delta M_{n(z_0,k),k} = M_{t_k^n} - M_{t_{n(z_0,k)}^n}$.

Since $b$ is strictly decreasing (because $b'(x)<0$ for all $x>0$) and $z_{i+1}^n > z_0$ for
$i \geq n(z_0,k)$, we have
\begin{equation}
b(z_{i+1}^n) \leq b(z_0) \quad \text{for all } i \geq n(z_0,k).
\end{equation}
Therefore,
\begin{equation}
\sum_{i=n(z_0,k)}^{k-1} b(z_{i+1}^n) \Delta t \leq b(z_0)(t_k^n - t_{n(z_0,k)}^n) \leq |b(z_0)|T.
\end{equation}
Moreover,
\begin{equation}
|\Delta M_{n(z_0,k),k}| \leq 2\|M\|_{\infty,T}.
\end{equation}
Hence,
\begin{equation}
0 < z_k^n \leq z_0 + |b(z_0)|T + 2\|M\|_{\infty,T},
\end{equation}
which shows that $(z_k^n)_{k=0,\dots,n}$ is uniformly bounded in $k$ and in $n$ (for fixed $T$).

\medskip\noindent Finally, we derive the convergence estimate.

Let $\xi_k^n := z_{t_k^n}$ be the exact solution evaluated at the grid points. The exact solution satisfies
\begin{equation}
\xi_{k+1}^n = \xi_k^n + \int_{t_k^n}^{t_{k+1}^n} b(z_s)\,\dd s + \big(M_{t_{k+1}^n} - M_{t_k^n}\big).
\end{equation}
Subtracting this from the scheme \eqref{eqnimplicit}, we obtain
\begin{equation}
z_{k+1}^n - \xi_{k+1}^n = z_k^n - \xi_k^n + \big[b(z_{k+1}^n) - b(\xi_{k+1}^n)\big]\Delta t + \varepsilon_k^n,
\end{equation}
where
\begin{equation}
\varepsilon_k^n := b(\xi_{k+1}^n)\Delta t - \int_{t_k^n}^{t_{k+1}^n} b(z_s)\,\dd s.
\end{equation}

Since $b$ is strictly decreasing, we distinguish two cases.

\medskip\noindent
\emph{Case 1:} If $z_{k+1}^n > \xi_{k+1}^n$, then $b(z_{k+1}^n) - b(\xi_{k+1}^n) \leq 0$, and thus
\begin{equation}
|z_{k+1}^n - \xi_{k+1}^n| = z_{k+1}^n - \xi_{k+1}^n 
\leq |z_k^n - \xi_k^n| + |\varepsilon_k^n|.
\end{equation}

\medskip\noindent
\emph{Case 2:} If $z_{k+1}^n \leq \xi_{k+1}^n$, then $b(\xi_{k+1}^n) - b(z_{k+1}^n) \leq 0$, and
\begin{equation}
|z_{k+1}^n - \xi_{k+1}^n| = \xi_{k+1}^n - z_{k+1}^n 
\leq |z_k^n - \xi_k^n| + |\varepsilon_k^n|.
\end{equation}

In both cases we have
\begin{equation}
|z_{k+1}^n - \xi_{k+1}^n| \leq |z_k^n - \xi_k^n| + |\varepsilon_k^n|.
\end{equation}
By induction it follows that
\begin{equation}
|z_k^n - \xi_k^n| \leq \sum_{i=0}^{k-1} |\varepsilon_i^n|,\qquad k=1,\dots,n.
\end{equation}

We now estimate the error term $\varepsilon_i^n$:
\begin{equation}
|\varepsilon_i^n| 
= \left| \int_{t_i^n}^{t_{i+1}^n} \big[b(\xi_{i+1}^n) - b(z_s)\big]\,\dd s \right|.
\end{equation}
Let $z_\ast := \inf_{t\in[0,T]} z_t$ and $z^\ast := \sup_{t\in[0,T]} z_t$. Since $z$ is continuous and strictly positive on $[0,T]$, we have $0<z_\ast\le z^\ast<\infty$. Since $b$ is continuously differentiable on
$[z_\ast,z^\ast]$, it is Lipschitz continuous on this interval with constant $L_b>0$. Therefore,
\begin{equation}
|\varepsilon_i^n| \leq L_b \int_{t_i^n}^{t_{i+1}^n} |\xi_{i+1}^n - z_s|\, \dd s.
\end{equation}
Since $z$ is $\alpha$-H\"older continuous on $[0,T]$ with $\alpha=\frac{1}{2}-\varepsilon$ (this follows from the fact
that $M$ has H\"older regularity $\alpha$ and the rough differential equation theory), we have
\[
|\xi_{i+1}^n - z_s| = |z_{t_{i+1}^n} - z_s| \leq \|z\|_{\alpha,T} (t_{i+1}^n - s)^\alpha.
\]
Hence,
\begin{equation}
|\varepsilon_i^n| \leq L_b \|z\|_{\alpha,T} \int_{t_i^n}^{t_{i+1}^n} (t_{i+1}^n - s)^\alpha \dd s
= L_b \|z\|_{\alpha,T} \frac{(\Delta t)^{\alpha+1}}{\alpha+1}.
\end{equation}

Thus,
\begin{equation}
\sum_{i=0}^{k-1} |\varepsilon_i^n| \leq
k \cdot L_b \|z\|_{\alpha,T} \frac{(\Delta t)^{\alpha+1}}{\alpha+1}
\leq L_b \|z\|_{\alpha,T} \frac{T^{\alpha+1}}{(\alpha+1)n^\alpha},
\end{equation}
and therefore, at grid points,
\begin{equation}
\max_{0 \le k \le n} |z_k^n - \xi_k^n| \leq
L_b \|z\|_{\alpha,T} \frac{T^{\alpha+1}}{(\alpha+1)n^\alpha}.
\end{equation}

Next, consider the continuous-time interpolation. For $t \in (t_k^n, t_{k+1}^n]$ we have
\begin{equation}
|z_t^n - z_t| \leq |z_t^n - z_k^n| + |z_k^n - \xi_k^n| + |\xi_k^n - z_t|.
\end{equation}

For the first term, by linear interpolation we have
\begin{equation}
\begin{aligned}
|z_t^n - z_k^n|
&\leq |z_{k+1}^n - z_k^n| \\
&\leq |z_{k+1}^n-\xi_{k+1}^n| + |\xi_{k+1}^n-\xi_k^n| + |\xi_k^n-z_k^n| \\
&\le 2\max_{0\le j\le n}|z_j^n-\xi_j^n| + \|z\|_{\alpha,T}\Delta t^\alpha.
\end{aligned}
\end{equation}
For the third term, using the H\"older regularity of $z$,
\begin{equation}
|\xi_k^n - z_t| = |z_{t_k^n} - z_t| \leq \|z\|_{\alpha,T} \Delta t^\alpha.
\end{equation}

Combining all estimates and using the bound for $|z_k^n - \xi_k^n|$, we obtain
\begin{equation}
\|z^n - z\|_{\infty,T} \leq C\, n^{-\alpha},
\end{equation}
for some (random) constant $C>0$ depending on $T$, $b$, $\|z\|_{\alpha,T}$, and $\|M\|_{\alpha,T}$. This completes the
proof.
\end{proof}

\subsection{Lower bound and optimal convergence rate for the singular equation driven by
mixed fBm}
In the previous subsection we established an upper convergence bound for the implicit Euler scheme applied to the singular equation. We now derive a corresponding lower bound for the approximation error from the skeleton of the driving process.

The lower bound identifies a universal polynomial barrier for uniform approximation. Together with the upper estimate, this determines the order of convergence of the scheme and yields optimality in the sense explained below.

Fix $T>0$. Let $M=B+B^{H}$ where $B$ is a standard Brownian motion, $B^{H}$ is an independent fractional Brownian motion with Hurst index $H>1/2$, and consider the singular equation
\begin{equation}\label{eq:singular}
  z_t = z_0 + \int_0^t b(z_s)\,ds + M_t,\qquad 
  b(x)=\Big(m+\tfrac12\Big)\frac{1}{x}-\frac{k}{2}x,\qquad z_0>0.
\end{equation}
(Here $m>0$ corresponds to the classical Feller condition $2k\theta>\sigma^2$ in the original CIR variables.)

Assume that the solution $z$ of \eqref{eq:singular} exists, is continuous, and satisfies $z_t>0$ for all $t\ge 0$ almost surely.
Then for each fixed $T>0$, the quantity
\[
  z_*(\omega) := \inf_{t\in[0,T]} z_t(\omega)
\]
is well-defined and strictly positive for almost every $\omega$ (pathwise statement): indeed, a continuous strictly positive function on the compact interval $[0,T]$ attains a strictly positive minimum.

\medskip

Let $t_k:=kT/n$, $\Delta t:=T/n$, and denote by $\mathcal{G}_n:=\sigma(M_{t_0},\dots,M_{t_n})$ the information contained in the sampled skeleton of $M$. The following two lemmas yield a simple ``information lower bound'' for approximating $z$ from the sampled skeleton $\mathcal{G}_n$

\begin{lem}\label{lem:link}
Let $z$ solve \eqref{eq:singular} and fix $T>0$. Define $z_*=\inf_{t\in[0,T]}z_t$ (random, but $z_*>0$ a.s.).
Let $\widehat z$ be any continuous process on $[0,T]$ (not necessarily solving any equation) such that
\[
  \|\widehat z-z\|_{\infty,T} \le \frac{z_*}{2},
  \qquad \text{where } \|f\|_{\infty,T}:=\sup_{t\in[0,T]}|f_t|.
\]
Define the associated ``reconstructed noise''
\begin{equation}\label{eq:Mhat}
  \widehat M_t := \widehat z_t - z_0 - \int_0^t b(\widehat z_s)\,ds,\qquad t\in[0,T].
\end{equation}
Then
\begin{equation}\label{eq:link-ineq}
  \|z-\widehat z\|_{\infty,T}\ \ge\ \frac{1}{1+L_*T}\ \|M-\widehat M\|_{\infty,T},
  \qquad 
  L_* := \frac{4(m+\tfrac12)}{z_*^2}+\frac{k}{2}.
\end{equation}
Moreover, if $\widehat z$ is $\mathcal{G}_n$-measurable, then $\widehat M$ is also $\mathcal{G}_n$-measurable.
\end{lem}

\begin{proof}
On the event $\{\| \widehat z-z\|_{\infty,T} \le z_*/2\}$ we have $\widehat z_t \ge z_t - z_*/2 \ge z_*/2$ for all $t\in[0,T]$.
Since $b'(x)=-(m+\tfrac12)x^{-2}-\frac{k}{2}$, we have for all $x\ge z_*/2$,
\[
  |b'(x)| \le \frac{m+\tfrac12}{(z_*/2)^2}+\frac{k}{2}=\frac{4(m+\tfrac12)}{z_*^2}+\frac{k}{2}=:L_*,
\]
hence $b$ is Lipschitz on $[z_*/2,\infty)$ with Lipschitz constant $L_*$.
Subtracting the identity \eqref{eq:Mhat} from \eqref{eq:singular} yields, for all $t\in[0,T]$,
\[
  (M_t-\widehat M_t) = (z_t-\widehat z_t) - \int_0^t\big(b(z_s)-b(\widehat z_s)\big)\,ds.
\]
Taking absolute values, using the Lipschitz bound for $b$ and the definition of $\|\cdot\|_{\infty,T}$ gives
\[
  |M_t-\widehat M_t|
  \le \|z-\widehat z\|_{\infty,T} + \int_0^t L_* \|z-\widehat z\|_{\infty,T}\,ds
  \le (1+L_*T)\|z-\widehat z\|_{\infty,T}.
\]
Taking the supremum over $t\in[0,T]$ yields
$\|M-\widehat M\|_{\infty,T}\le (1+L_*T)\|z-\widehat z\|_{\infty,T}$, which is equivalent to \eqref{eq:link-ineq}.

Finally, if $\widehat z$ is $\mathcal{G}_n$-measurable, then \eqref{eq:Mhat} shows $\widehat M$ is obtained from $\widehat z$ by deterministic operations (time integration of a measurable integrand), hence $\widehat M$ is $\mathcal{G}_n$-measurable.
\end{proof}

\begin{lem}\label{lem:noise-lb}
Let $t_k=kT/n$ and $\mathcal{G}_n=\sigma(M_{t_0},\dots,M_{t_n})$.
For any $\mathcal{G}_n$-measurable (real-valued) random variable $\widehat M_s$ at a fixed time $s\in(0,T)$, there exist absolute constants $c_0,p_0\in(0,1)$ (independent of $n$ and of the choice of $\widehat M_s$) such that
\[
  \mathbb{P}\Big(|M_s-\widehat M_s|\ \ge\ c_0\sqrt{\Delta t}\Big)\ \ge\ p_0.
\]
In particular, for any continuous $\mathcal{G}_n$-measurable process $\widehat M$ on $[0,T]$,
\begin{equation}\label{eq:noise-sup-lb}
  \mathbb{P}\Big(\|M-\widehat M\|_{\infty,T}\ \ge\ c_0\sqrt{\Delta t}\Big)\ \ge\ p_0.
\end{equation}
\end{lem}

\begin{proof}
Fix the first interval $[t_0,t_1]=[0,\Delta t]$ and take $s:=\Delta t/2$.
Decompose the Brownian motion $B$ on $[0,\Delta t]$ as linear interpolation plus an independent Brownian bridge term:
\[
  B_s = \frac{B_0+B_{\Delta t}}{2} + \beta,
  \qquad 
  \beta := B_{\Delta t/2}-\frac{B_0+B_{\Delta t}}{2}.
\]
It is standard that $\beta\sim\mathcal{N}(0,\Delta t/4)$ and that $\beta$ is independent of the entire Brownian skeleton $\sigma(B_{t_0},\dots,B_{t_n})$
(and in particular independent of $\sigma(B_0,B_{\Delta t})$); moreover, since $B^H$ is independent of $B$, $\beta$ is independent of the sigma-field generated by $\{B^H_t: t\in[0,T]\}$.
Therefore, $\beta$ is independent of $\mathcal{G}_n=\sigma(M_{t_0},\dots,M_{t_n})$.

Write
\[
  M_s = B_s + B^H_s
      = \underbrace{\frac{B_0+B_{\Delta t}}{2} + B^H_s}_{=:U} + \beta.
\]
The random variable $U$ may depend on $(B_0,B_{\Delta t})$ and on $B^H_s$, but $\beta$ is independent of $(U,\mathcal{G}_n)$.
Let $\widehat M_s$ be any $\mathcal{G}_n$-measurable random variable. Conditioning on $(U,\mathcal{G}_n)$, the quantity $(U-\widehat M_s)$ becomes a constant and
\[
  M_s-\widehat M_s = \beta + (U-\widehat M_s).
\]
For any $a>0$ and any constant $c\in\mathbb{R}$, the probability $\mathbb{P}(|\beta+c|\ge a)$ is minimized at $c=0$ (shifting a centered normal away from $0$ can only increase the mass outside $[-a,a]$).
Hence, almost surely,
\[
  \mathbb{P}\big(|M_s-\widehat M_s|\ge a \,\big|\, U,\mathcal{G}_n\big)
  = \mathbb{P}\big(|\beta+(U-\widehat M_s)|\ge a \,\big|\, U,\mathcal{G}_n\big)
  \ge \mathbb{P}(|\beta|\ge a).
\]
Choose $a=c_0\sqrt{\Delta t}$ with, say, $c_0:=1/8$. Since $\beta/\sqrt{\Delta t}\sim\mathcal{N}(0,1/4)$, the quantity
$p_0:=\mathbb{P}(|\beta|\ge c_0\sqrt{\Delta t})$ is a fixed positive constant independent of $n$.
Taking expectations over $(U,\mathcal{G}_n)$ yields
\[
  \mathbb{P}\Big(|M_s-\widehat M_s|\ge c_0\sqrt{\Delta t}\Big)\ge p_0,
\]
proving the first claim. The sup-norm bound \eqref{eq:noise-sup-lb} follows since $\|M-\widehat M\|_{\infty,T}\ge |M_s-\widehat M_s|$.
\end{proof}

With these two lemmas and the up bound of the error of implicit Euler scheme we can illustrate the optimal convergence rate in probability: 

\begin{prop}\label{prop:halfbarrier}
Let $z$ solve \eqref{eq:singular} and fix $T>0$.
Let $(\widehat z^{\,n})_{n\ge 1}$ be any sequence of continuous $\mathcal{G}_n$-measurable processes on $[0,T]$.
Then for every $\delta>0$,
\[
  n^{\frac12+\delta}\,\|z-\widehat z^{\,n}\|_{\infty,T}\ \not\longrightarrow\ 0
  \qquad \text{in probability as } n\to\infty.
\]
Equivalently, no method based only on the skeleton $\{M_{t_k}\}_{k=0}^n$ can achieve a polynomial rate strictly better than $n^{-1/2}$ in the uniform norm (in probability).
\end{prop}

\begin{proof}
Assume by contradiction that there exist $\delta>0$ and a sequence $(\widehat z^{\,n})_{n\ge 1}$ such that
\begin{equation}\label{eq:contr-halfbarrier}
  n^{\frac12+\delta}\,\|z-\widehat z^{\,n}\|_{\infty,T}\longrightarrow 0
  \qquad\text{in probability as } n\to\infty.
\end{equation}
In particular, $\|z-\widehat z^{\,n}\|_{\infty,T}\to 0$ in probability.

We first justify the random threshold $z_*/2$. For $N\ge 1$ set $A_N:=\{z_*\ge 1/N\}$. Since $z_*>0$ almost surely,
$\mathbb{P}(A_N)\uparrow 1$ as $N\to\infty$. For each fixed $N$,
\[
  \mathbb{P}\Big(\|z-\widehat z^{\,n}\|_{\infty,T}>\tfrac{z_*}{2}\Big)
  \le \mathbb{P}(A_N^c)+\mathbb{P}\Big(\|z-\widehat z^{\,n}\|_{\infty,T}>\tfrac{1}{2N}\Big)
  \xrightarrow[n\to\infty]{}\ \mathbb{P}(A_N^c),
\]
and letting $N\to\infty$ yields
\begin{equation}\label{eq:threshold-prob}
  \mathbb{P}\Big(\|z-\widehat z^{\,n}\|_{\infty,T}\le \tfrac{z_*}{2}\Big)\longrightarrow 1.
\end{equation}

Define $\widehat M^{\,n}$ from $\widehat z^{\,n}$ by \eqref{eq:Mhat} (for all $\omega$). On the event
$\{\|z-\widehat z^{\,n}\|_{\infty,T}\le z_*/2\}$ we can apply Lemma~\ref{lem:link}.
Lemma \ref{lem:link} yields
\[
  \|M-\widehat M^{\,n}\|_{\infty,T}\ \le\ (1+L_*T)\,\|z-\widehat z^{\,n}\|_{\infty,T},
\]
where $L_*=\frac{4(m+\tfrac12)}{z_*^2}+\frac{k}{2}$ depends only on the (random) lower bound $z_*$.
Combining this with \eqref{eq:contr-halfbarrier}, the fact that $L_*<\infty$ almost surely, and
\eqref{eq:threshold-prob}, we obtain
\[
  n^{\frac12+\delta}\,\|M-\widehat M^{\,n}\|_{\infty,T}\longrightarrow 0
  \qquad\text{in probability as } n\to\infty.
\]

This contradicts Lemma \ref{lem:noise-lb}, which gives constants $c_0,p_0>0$ such that for every $n$,
\[
  \mathbb{P}\Big(\|M-\widehat M^{\,n}\|_{\infty,T} \ge c_0\sqrt{\Delta t}\Big)\ \ge\ p_0,
  \qquad \sqrt{\Delta t}=\sqrt{T/n}.
\]
Indeed, if $n^{\frac12+\delta}\|M-\widehat M^{\,n}\|_{\infty,T}\to 0$ in probability, then in particular
$\mathbb{P}\big(\|M-\widehat M^{\,n}\|_{\infty,T}\ge c_0\sqrt{T/n}\big)\to 0$, a contradiction.
\end{proof}

\begin{remark}
Proposition \ref{prop:halfbarrier} shows that the exponent $1/2$ is a universal \emph{polynomial barrier} for uniform approximation of $z$ from the skeleton of $M$ (in probability).
If one has an upper bound of the form $\|z^n-z\|_{\infty,T}\le C(\omega)\,n^{-(1/2-\varepsilon)}$ (as in the implicit Euler analysis),
then the scheme is \emph{near-optimal} in the sense that it attains any polynomial rate strictly smaller than the barrier $1/2$.    
\end{remark}

\begin{remark}[why the exponent $1/2$ appears in the pathwise rate]
Writing $\Delta t:=T/n$, the upper bound in Theorem~\ref{thm:euler-convergence} reads (pathwise)
$\|z^n-z\|_{\infty,T}\le C(\omega)\,(\Delta t)^{\alpha}$ for every $\alpha<1/2$, i.e.\ $\|z^n-z\|_{\infty,T}=O_\omega(n^{-\alpha})$.
One cannot take $\alpha=1/2$ because the Brownian part of $M$ is H\"older continuous only for exponents strictly smaller than $1/2$.
On the other hand, Proposition~\ref{prop:halfbarrier} shows that no algorithm based solely on the sampled skeleton can beat the scale $\sqrt{\Delta t}$ (equivalently $n^{-1/2}$) in the uniform norm, due to the Brownian bridge fluctuation inside each grid interval; the fractional component $B^H$ is smoother ($H>1/2$) and does not change this barrier.
In log--log plots, a rate $n^{-1/2}$ corresponds to slope $-1/2$, which explains the notation ``$-1/2$'' for the optimal exponent.
\end{remark}

\begin{remark}[$L^p(\Omega)$-strong error: what is missing]
Our error bound in Theorem~\ref{thm:euler-convergence} is \emph{pathwise} and involves a random constant $C(\omega)$ depending on (at least) the H\"older norms of $M$ and on inverse powers of the random minimum $z_*=\inf_{t\in[0,T]}z_t$ (see also $L_*=\frac{4(m+\tfrac12)}{z_*^2}+\frac{k}{2}$ in Lemma~\ref{lem:link}).
To upgrade to an $L^p(\Omega)$-strong estimate one would need suitable moment bounds such as $\mathbb{E}[C(\omega)^p]<\infty$, which is delicate near the Feller boundary.
This prevents a direct extension of the techniques in \cite{HONG20202675,10.1098/rspa.2011.0505} and requires additional structural arguments (possibly of Malliavin/type or precise inverse-moment estimates), which we leave for future work.
\end{remark}

\subsection{The Simulations of The Mixed CIR Model}

In the following, we illustrate the behavior of the mixed fractional CIR model
and the implicit Euler scheme. We simulate the short rate process $r_t$ on
the interval $[0,10]$ with different numbers of sample paths.

\begin{center}
\resizebox{160mm}{80mm}{\includegraphics{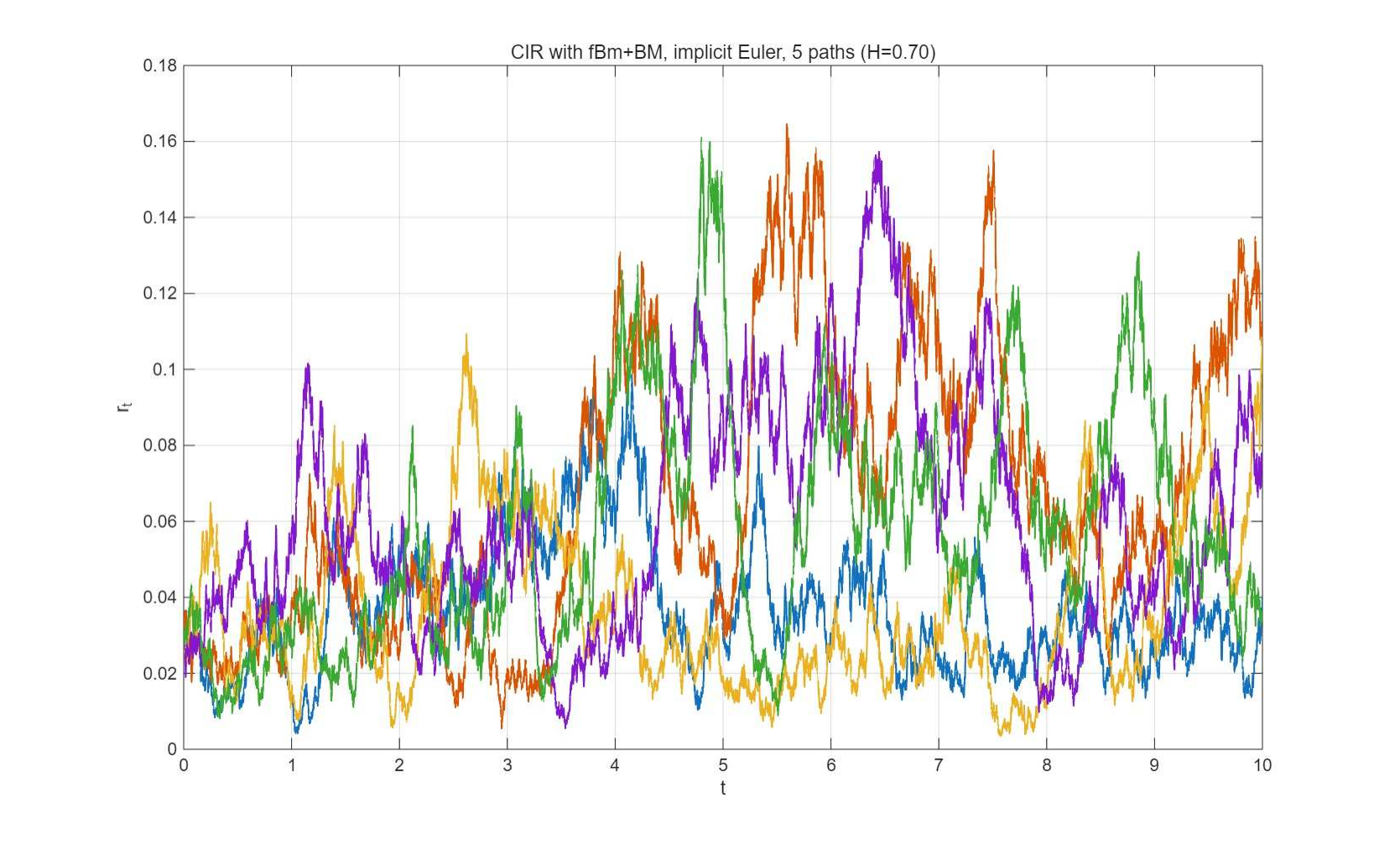}}\\
\resizebox{160mm}{80mm}{\includegraphics{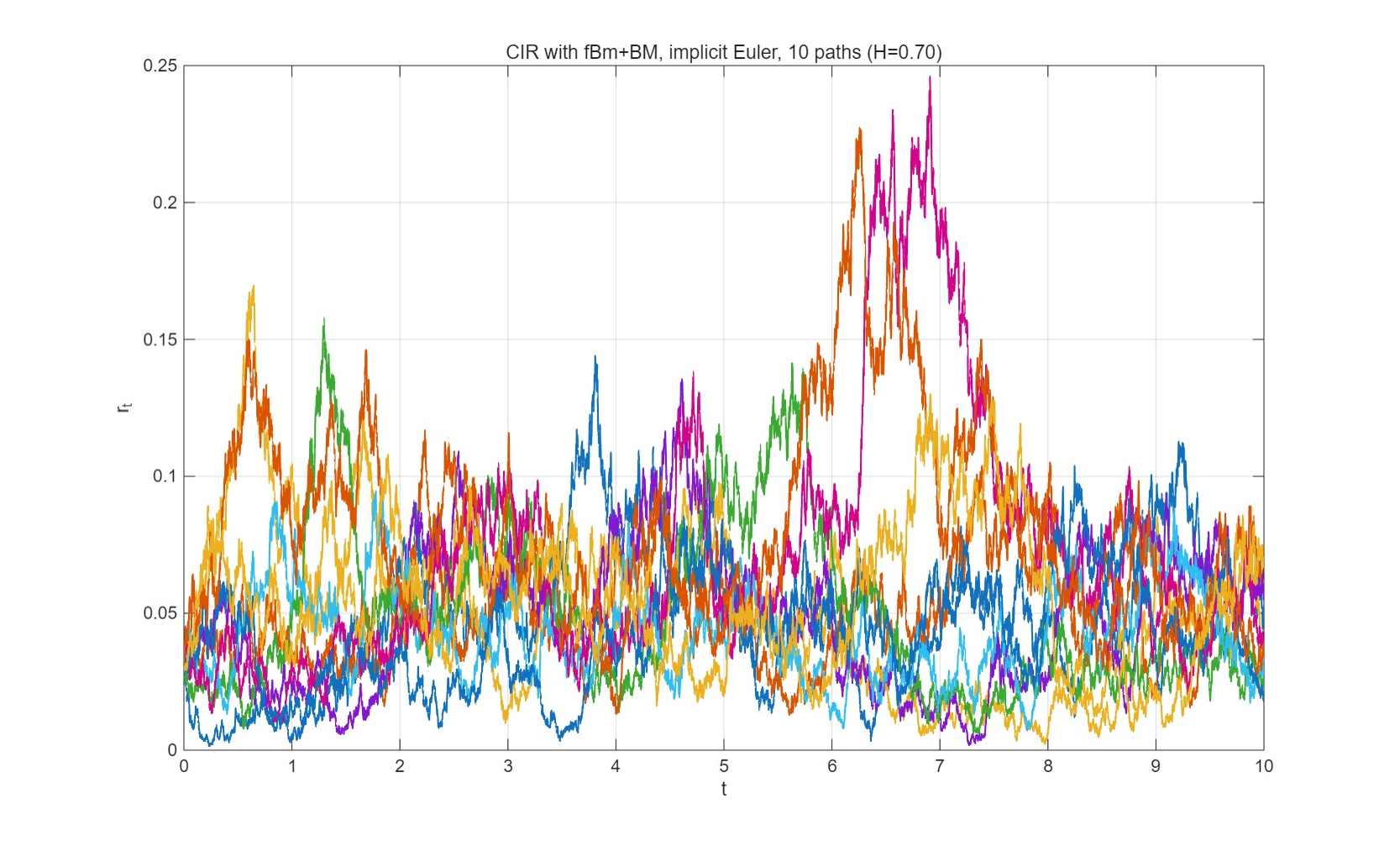}}\\
\resizebox{160mm}{80mm}{\includegraphics{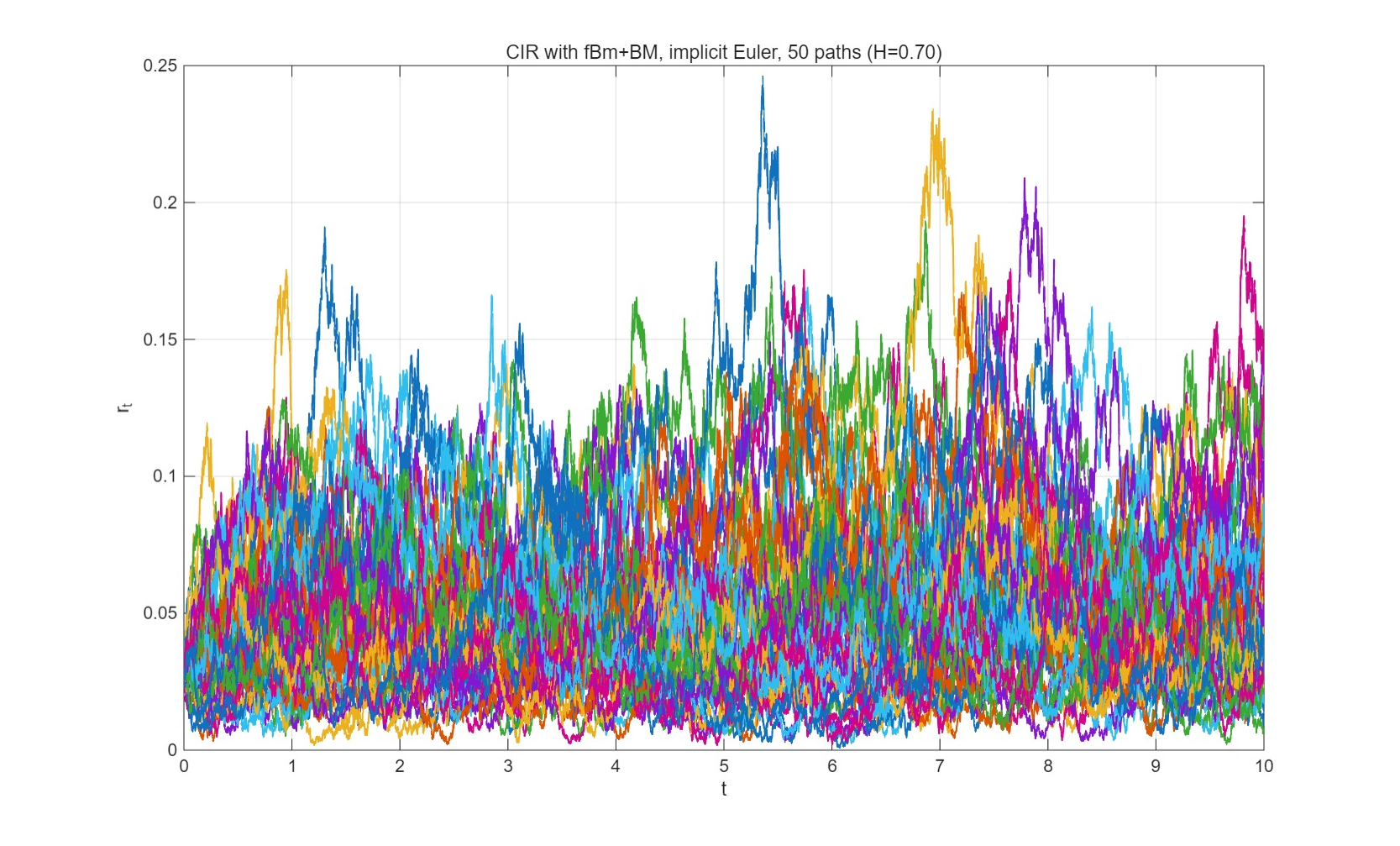}}
Fig.1.Implicit Euler simulations of the mixed fractional CIR short rate
  on the interval $[0,10]$. The figure shows trajectories of $r_t$ obtained
  from the implicit Euler scheme for $5$, $10$ and $50$ independent sample paths.
\end{center}

\end{document}